\documentclass[11pt]{article}

\usepackage{amsmath}
\usepackage{amssymb}
\usepackage{amsthm}
\usepackage{indentfirst}
\usepackage{tikz-cd}

\usepackage{xcolor}

\usepackage{hyperref}
\hypersetup{colorlinks=true,linkcolor=red,pdfpagemode=UseNone,pdfstartview=FitH}

\newcommand{\Z}{\ensuremath{\mathbb{Z}}}
\newcommand{\Q}{\ensuremath{\mathbb{Q}}}

\usepackage[margin=2.6cm]{geometry}

\theoremstyle{plain}
\newtheorem{theorem}{Theorem}[section]
\newtheorem{lemma}[theorem]{Lemma}
\newtheorem{corollary}[theorem]{Corollary}
\newtheorem{proposition}[theorem]{Proposition}

\theoremstyle{definition}

\newtheorem{remark}[theorem]{Remark}

\DeclareMathOperator{\Hom}{Hom}
\DeclareMathOperator{\Gal}{Gal}

\DeclareMathOperator{\ord}{ord}
\DeclareMathOperator{\ab}{ab}

\DeclareMathOperator{\rest}{res}
\DeclareMathOperator{\infl}{inf}
\DeclareMathOperator{\red}{red}

\DeclareMathOperator{\img}{im}
\DeclareMathOperator{\coker}{coker}
\DeclareMathOperator{\lcm}{lcm}

\DeclareMathOperator{\norm}{N}
\DeclareMathOperator{\rad}{rad}

\newcommand{\ff}{f'}

\begin{document}

\title{On the Chevalley-Bass number of a field}

\author{Jean Gillibert \and Florence Gillibert \and Gabriele Ranieri}

\date{November 2026}

\maketitle

%%%%%%%%%%%%%%%%%%%%%%%%%%%%%%%%%%%%%%%%%%%%%

\begin{abstract}
We give upper and lower bounds on the Chevalley-Bass number of a field of characteristic zero, whenever this quantity is well-defined. We also describe an algorithm which computes the Chevalley-Bass number of a field, provided its maximal abelian subextension is known. As a primary application, we improve the value of a constant related to exponential Diophantine equations. 
\end{abstract}

%%%%%%%%%%%%%%%%%%%%%%%%%%%%%%%%%%%%%%%%%%%%%

\section{Introduction}

Let $K$ be a field of characteristic $0$ such that $K_{\ab}/\Q$, the maximal abelian subextension of $K/\Q$, has finite degree. It follows from works of Chevalley \cite{chevalley1951}, Bass \cite{bass1965} and Smith \cite{smith1970} that there exists a positive integer $\Lambda$ such that
\begin{equation}
\label{eq:chbassnumber}
\text{for any positive integer $n$}, \quad (K(\zeta_{\Lambda n})^\times)^{\Lambda n}\cap K^\times \subseteq (K^\times)^n.
\end{equation}

We refer the reader to Bilu's article \cite{bilu2023} for a detailed exposition of this result, motivated by the study of exponential Diophantine equations---the proof given there follows closely Smith's argument. Following the terminology introduced by Bilu, the smallest positive integer $\Lambda$ satisfying property \eqref{eq:chbassnumber} is called the \emph{Chevalley-Bass number} of $K$. As pointed out by Bilu, the positive integers satisfying \eqref{eq:chbassnumber} are exactly the multiples of the Chevalley-Bass number.

\medskip

We shall prove the following result.

\begin{theorem}
\label{thm:main}
Let $K$ be a field of characteristic $0$ such that $K_{\ab}/\Q$, the maximal abelian subextension of $K/\Q$, has finite degree. Let $\lambda$ be the number of roots of unity contained in $K$, let $f$ be the conductor of $K_{\ab}/\Q$, and let
$$
\ff:=
\begin{cases}
\prod_{p\mid \lambda} p^{\ord_p(f)} & \text{if $4\mid f$} \\
4\prod_{p\mid \lambda} p^{\ord_p(f)} & \text{if $f$ is odd.}
\end{cases}
$$
Then the Chevalley-Bass number $\Lambda$ of $K$ satisfies
$$
\lcm\left(4,\lambda,\frac{\ff}{\lambda}\right)\mid \Lambda \mid \ff.
$$
\end{theorem}

It follows from Theorem~\ref{thm:main} that the Chevalley-Bass number of $K$ has exactly the same prime factors as $\lambda$, the number of roots of unity in $K$. In particular, if $K_{\ab}/\Q$ is totally real then the Chevalley-Bass number of $K$ is a power of $2$.

An easy consequence of our main result (Corollary~\ref{cor:degree_bound}) is an upper bound for the Chevalley-Bass number $\Lambda$ of $K$:
$$
\Lambda \leq 2[K_{\ab}:\Q]\left(2+\log_2[K_{\ab}:\Q]\right).
$$

Actually, our technique yields an algorithm which allows one to compute the Chevalley-Bass number of a field, provided $K_{\ab}/\Q$ is known (see \S{}\ref{sub:algorithm}). In order to illustrate this,  we construct in \S{}\ref{sub:examples}, for any prime $p\geq 3$ and suitable values of $m$, fields of conductor $f=p^4m$ and number of roots of unity $\lambda=2p^2$, such that $\Lambda$ takes all possible values  in Theorem~\ref{thm:main}, namely: $4p^2$, $4p^3$ and $4p^4$. One can construct, by the same strategy, similar examples in which $f$ and $\lambda$ have arbitrary size.

Our results allow us to answer the questions asked by Bilu \cite[\S{}6.3]{bilu2023}:
\begin{enumerate}
\item The fields $\Q$ and $\Q(\zeta_4)$ both have Chevalley-Bass number $4$.
\item There exists an algorithm which computes the Chevalley-Bass number of a given number field $K$. The same applies to an arbitrary field provided $K_{\ab}/\Q$ is known.
\item The Chevalley-Bass number of a field is always a multiple of $4$.
\item Conversely, any multiple of $4$ is the Chevalley-Bass number of some cyclotomic field. Indeed, if $4\mid r$ then the Chevalley-Bass number of $\Q(\zeta_r)$ is $r$. We also observe that if $r$ is odd then the Chevalley-Bass number of $\Q(\zeta_r)$ is $4r$.
\item If $K$ is the splitting field of a polynomial of degree $m\geq 2$ with rational coefficients, then
$$
\Lambda\leq 2\cdot 3^{m/3}\left(2+\frac{m}{3}\log_23\right).
$$
\end{enumerate}

Unsurprisingly, our proof is based on Galois cohomology computations, a strategy that was already used by Laurent \cite{laurent1984} to tackle a similar  problem. To start, we give in \S{}\ref{sub:cohomologicalcriterion} a cohomological characterisation of the Chevalley-Bass number of $K$ (see Corollary~\ref{cor:prime_by_prime}), and prove that $K$ and $K_{\ab}$ have the same Chevalley-Bass number. Then we focus on general properties of the cohomology groups and maps involved. One notable consequence of our computations is the following (see Lemma~\ref{lem:exponent}).

\begin{proposition}
\label{prop:intro}
Under the assumptions of Theorem~\ref{thm:main}, for any positive integer $n$ the Galois cohomology group
$$
H^1(\Gal(K(\zeta_n)/K),\mu_n)
$$
is killed by $\lambda$. Moreover, it has exponent $\lambda$ for infinitely many values of $n$.
\end{proposition}

Actually, we also prove that the order of this group is unbounded when $K$ is fixed and $n$ varies; in particular \cite[Lemma~2]{laurent1984} is not correct as stated, but this does not affect the rest of the paper since the key point is the boundedness of the exponent.

It follows from Proposition~\ref{prop:intro} that $\lambda$ satisfies the following property, slightly weaker than \eqref{eq:chbassnumber}.

\begin{corollary}
\label{cor:intro}
For any positive integer $n$, the following holds: if $x\in K^\times$ is a $\lambda{}n$th power in $K(\zeta_{\lambda n})$ then $x=\varepsilon y^n$ for some $y\in K$ and some $\lambda$th root of unity $\varepsilon\in K$.
\end{corollary}

In fact, Corollary~\ref{cor:intro} can also be deduced from Schinzel's theorem on abelian binomials \cite{schinzel1977}; our cohomological proof is independent.

This allows us to improve on Proposition~4.4 of \cite{BLNOW2025}, which is an explicit version of a result of Laurent \cite{laurent1984}.

\begin{corollary}
\label{cor:LaurentBilu}
Let $K$ be a finitely generated field of characteristic $0$, let $\lambda$ be the number of roots of unity contained in $K$, and let
$$
\Gamma:= \{x\in \overline{K}^\times ~|~ x^n\in K^\times ~\text{for some}~ n>0\}.
$$
Let $\alpha_1,\dots,\alpha_s$ be elements of $\Gamma$ such that $\alpha_1+\cdots +\alpha_s=1$ and no proper subsum of $\alpha_1+\cdots +\alpha_s$ vanishes. Then there exist roots of unity $\xi_1,\dots,\xi_s$ in $\overline{K}^\times$ such that $\alpha_i^\lambda\xi_i$ belongs to $K$.
\end{corollary}

\begin{remark}
\label{remark:forYuri}
As a consequence of Corollary~\ref{cor:LaurentBilu}, one can replace in Proposition~4.5 of \cite{BLNOW2025} the Chevalley-Bass number $\Lambda$ of $K$ by the number $\lambda$ of roots of unity contained in $K$. Likewise, one can replace $\rho\exp\exp(m/\log m)$ by $2\rho\exp(m)$ in Proposition~4.7 of \emph{loc. cit.} Details are given after the proof of Corollary~\ref{cor:LaurentBilu}.

The same remark applies to \cite{laurent1984}: in \S{}2.2 one can take the exponent $b$ to be the number of roots of unity contained in $K$.
\end{remark}

In order to keep the paper self-contained, we give full details. The ingredients are elementary: Galois theory of cyclotomic extensions, cohomology of cyclic groups, and inflation-restriction sequences.

%%%%%%%%%%%%%%%%%%%%%%%%%%%%%%%%%%%%%%%%%%%%%
%%%%%%%%%%%%%%%%%%%%%%%%%%%%%%%%%%%%%%%%%%%%%

\section{Proofs}

From now on, we consider a field  $K$ of characteristic $0$, with the property that $K_{\ab}/\Q$, the maximal abelian subextension of $K/\Q$, has finite degree.
This condition is trivially satisfied if $K$ is a number field, and more generally if $K$ is a finitely generated field.

For any positive integer $n$, we denote by $\mu_n$ the cyclic group of $n$th roots of unity in an algebraic closure of $K$, endowed with its natural Galois action. We denote by $\zeta_n$ a primitive $n$th root of unity, and we let
$$
G_n := \Gal(K(\zeta_n)/K).
$$

It is well-known that $G_n$ is isomorphic to a subgroup of $(\Z/n\Z)^\times$, which acts by exponentiation on the set of $n$th roots of unity.

The Chinese remainder Theorem induces a canonical isomorphism
$$
(\Z/n\Z )^\times \simeq  \prod_{p\mid n} (\Z/p^{\ord_p(n)}\Z)^\times. 
$$

Finally, we denote by $\lambda$ the number of roots of unity contained in $K$. Equivalently, $\lambda$ is the largest integer with the property that $\Q(\zeta_\lambda)\subseteq K$.

%%%%%%%%%%%%%%%%%%%%%%%%%%%%%%%%%%%%%%%%%%%%%

\subsection{Cohomological characterisation of the Chevalley-Bass number}
\label{sub:cohomologicalcriterion}

In this section we prove that the Chevalley-Bass number can be characterised purely in terms of the cohomology of the group $G_n$. As a by-product, we deduce that the Chevalley-Bass number of $K$ is equal to that of $K_{\ab}$.

\begin{lemma}
\label{lem:cohomologicalCB}
\begin{enumerate}
\item[(i)] The Chevalley-Bass number of $K$ is the smallest integer $\Lambda$ such that, for every integer $n>0$, the image of the inflation map
$$
\infl:H^1(G_{\Lambda{}n},\mu_{\Lambda{}n}) \to H^1(K,\mu_{\Lambda{}n})
$$
lies in the image of the map $H^1(K,\mu_{\Lambda}) \to H^1(K,\mu_{\Lambda{}n})$ induced by the inclusion $\mu_{\Lambda} \subseteq \mu_{\Lambda{}n}$.
\item[(ii)] The Chevalley-Bass number of $K$ is the smallest integer $\Lambda$ such that, for every integer $n>0$, the map $H^1(G_{\Lambda{}n},\mu_{\Lambda}) \to H^1(G_{\Lambda{}n},\mu_{\Lambda{}n})$ induced by the inclusion $\mu_{\Lambda} \subseteq \mu_{\Lambda{}n}$, is surjective.
\end{enumerate}
\end{lemma}

\begin{proof}
Let $\Lambda$ and $n$ be two positive integers. Consider the commutative diagram
\begin{equation}
\label{eq:maindiagram}
\begin{tikzcd}
0 \arrow[r] 
    & H^1(G_{\Lambda{}n},\mu_{\Lambda})
    \arrow[r, "\infl"] \arrow[d]
    & H^1(K,\mu_{\Lambda})
    \arrow[r, "\rest"] \arrow[d]
    & H^1(K(\zeta_{\Lambda{}n}),\mu_{\Lambda})
    \arrow[d] \\
  0 \arrow[r] 
    & H^1(G_{\Lambda{}n},\mu_{\Lambda{}n}) 
        \arrow[r, "\infl"] 
    & H^1(K,\mu_{\Lambda{}n}) 
        \arrow[r, "\rest"] \arrow[d] 
    & H^1(K(\zeta_{\Lambda{}n}),\mu_{\Lambda{}n}) \\
  & & H^1(K,\mu_n) 
\end{tikzcd}
\end{equation}
in which the horizontal sequences are inflation-restriction sequences, and the vertical sequence in the middle is part of the long exact sequence of cohomology associated to the short exact sequence
$$
1 \to \mu_{\Lambda} \to \mu_{\Lambda{}n} \to \mu_n \to 1.
$$

We observe that the vertical right-hand side map in \eqref{eq:maindiagram} is injective, because the absolute Galois group of $K(\zeta_{\Lambda{}n})$ acts trivially on $\mu_{\Lambda{}n}$ and $\mu_{\Lambda}$, hence the $H^1$'s involved are $\Hom$'s. It then follows from diagram chase in \eqref{eq:maindiagram} that the following conditions are equivalent:
\begin{enumerate}
\item[(a)] the kernel of restriction $H^1(K,\mu_{\Lambda{}n})\to H^1(K(\zeta_{\Lambda{}n}),\mu_{\Lambda{}n})$ is contained in the kernel of the map $H^1(K,\mu_{\Lambda{}n}) \to H^1(K,\mu_n)$.
\item[(b)] the image of inflation $H^1(G_{\Lambda{}n},\mu_{\Lambda{}n}) \to H^1(K,\mu_{\Lambda{}n})$ is contained in the image of the map $H^1(K,\mu_{\Lambda}) \to H^1(K,\mu_{\Lambda{}n})$.
\item[(c)] the map $H^1(G_{\Lambda{}n},\mu_{\Lambda}) \to H^1(G_{\Lambda{}n},\mu_{\Lambda{}n})$ is surjective.
\end{enumerate}

On the other hand, for any field $F$ of characteristic zero and any positive integer $m$ we have, by Kummer theory, a canonical isomorphism
$$
H^1(F,\mu_m)\simeq F^\times/(F^\times)^m.
$$
Therefore, condition (a) is equivalent to the following condition:
\begin{enumerate}
\item[(a')] the kernel of the base-change map $K^\times/(K^\times)^{\Lambda{}n} \to K(\zeta_{\Lambda{}n})^\times/(K(\zeta_{\Lambda{}n})^\times)^{\Lambda{}n}$
is contained in the kernel of the natural map $K^\times/(K^\times)^{\Lambda{}n} \to K^\times/(K^\times)^n$.
\end{enumerate}
which, in turn, is equivalent to
\begin{enumerate}
\item[(a'')] $(K(\zeta_{\Lambda{}n})^\times)^{\Lambda{}n}\cap K^\times \subseteq (K^\times)^n$.
\end{enumerate}

To conclude, the Chevalley-Bass number of $K$ is by definition the smallest positive integer $\Lambda$ satisfying condition (a'') for all $n>0$. The equivalence with (b) proves (i), and the equivalence with (c) proves (ii). 
\end{proof}

\begin{corollary}
\label{cor:CB_is_ab}
The fields $K$ and $K_{\ab}$ have the same Chevalley-Bass number.
\end{corollary}

\begin{proof}
Let $\bar{K}$ be an algebraic closure of $K$. Given an integer $n>0$, there is (by Galois theory) a unique copy of $\Q(\zeta_n)$ in $\bar{K}$. So the intersection $\Q(\zeta_n)\cap K$ is well defined. Furthermore, since $\Q(\zeta_n)\cap K$ is a subfield of $\Q(\zeta_n)$, it is an abelian extension of $\Q$, hence it is contained in $K_{\ab}$. It follows that
$$
\Q(\zeta_n)\cap K = \Q(\zeta_n)\cap K_{\ab}.
$$

Then, by Galois theory, we have canonical isomorphisms
\begin{align*}
\Gal(K(\zeta_n)/K) &\simeq \Gal(\Q(\zeta_n)/\Q(\zeta_n)\cap K) \\
                   &= \Gal(\Q(\zeta_n)/\Q(\zeta_n)\cap K_{\ab}) \\
                   &\simeq \Gal(K_{\ab}(\zeta_n)/K_{\ab}).
\end{align*}
This yields a canonical isomorphism
$$
H^1(\Gal(K(\zeta_n)/K),\mu_n) \simeq H^1(\Gal(K_{\ab}(\zeta_n)/K_{\ab}),\mu_n)
$$
and the result follows from the cohomological characterisation of the Chevalley-Bass number (Lemma~\ref{lem:cohomologicalCB}(ii)).
\end{proof}

To conclude the section, we characterise the $p$-adic valuation of the Chevalley-Bass number. This allows us to handle one prime at a time.

\begin{corollary}
\label{cor:prime_by_prime}
Let $\Lambda$ be the Chevalley-Bass number of $K$, and let $p$ be a prime number. Then $\ord_p(\Lambda)$ is the smallest integer $j$ such that, for every integer $n>0$, the map
\begin{equation}
\label{eq:prime_by_prime_surj}
H^1(G_{p^jn},\mu_{p^j}) \to H^1(G_{p^jn},\mu_{p^{j+\ord_p(n)}})
\end{equation}
induced by the inclusion $\mu_{p^j} \subseteq \mu_{p^{j+\ord_p(n)}}$, is surjective. 
\end{corollary}

\begin{proof}
Let $\Lambda_0$ be the product of all $p^j$ where $p$ runs through prime numbers, and $j$ is the smallest integer satisfying the condition above.
Then for every $n>0$ we have a commutative square
\begin{equation}
\label{eq:chinese_square}
\begin{tikzcd}
H^1(G_{\Lambda_0{}n},\mu_{\Lambda_0}) \arrow[r, "\sim"] \arrow[d]
& \prod_p H^1(G_{\Lambda_0{}n}, \mu_{p^j}) \arrow[d] \\
H^1(G_{\Lambda_0{}n},\mu_{\Lambda_0{n}})  \arrow[r, "\sim"]
&  \prod_p H^1(G_{\Lambda_0{}n},\mu_{p^{j+\ord_p(n)}})
\end{tikzcd}
\end{equation}
in which the products run through all prime numbers, and horizontal maps are isomorphisms by the Chinese remainder Theorem. If the map \eqref{eq:prime_by_prime_surj} is surjective for all $n$ and $p$, then the vertical right-hand side map is surjective on each component, hence the left-hand side map is surjective for all $n$. It follows from Lemma~\ref{lem:cohomologicalCB}(ii) that $\Lambda\mid \Lambda_0$.

Conversely, fix a prime number $p$, let $i:=\ord_p(\Lambda)$, and let us write $\Lambda=rp^i$ with $p\nmid r$. Consider the diagram \eqref{eq:chinese_square} in which $\Lambda_0$ is replaced by $\Lambda$; then by Lemma~\ref{lem:cohomologicalCB}(ii) the vertical left-hand side map is surjective, hence the vertical right-hand side map is surjective on each component, in particular
$$
\theta:H^1(G_{\Lambda{}n}, \mu_{p^i}) \to H^1(G_{\Lambda{}n},\mu_{p^{i+\ord_p(n)}})
$$
is surjective. By Galois theory, we have an exact sequence
\begin{equation*}
\begin{tikzcd}
0 
  \arrow[r] 
& T
  \arrow[r] 
& G_{rp^in}(=G_{\Lambda{}n})
  \arrow[r] 
& G_{p^in}
  \arrow[r] 
& 0
\end{tikzcd}
\end{equation*}
whose kernel $T$ acts trivially on $\mu_{p^{i+\ord_p(n)}}$ by construction.
The inflation-restriction exact sequence for the Galois action on $\mu_{p^i}$ (resp. $\mu_{p^{i+\ord_p(n)}}$) yields the top (resp. bottom) exact line in the commutative diagram
$$
\begin{tikzcd}
0 
  \arrow[r] 
& H^1(G_{p^in},\mu_{p^i})
  \arrow[r, "\infl_1"] 
  \arrow[d, "\rho"]
& H^1(G_{rp^in},\mu_{p^i})
    \arrow[r, "\rest_1"]
    \arrow[d, "\theta"]
& H^1(T,\mu_{p^i})^{G_{p^in}}
    \arrow[d]
 \\
0 
  \arrow[r] 
& H^1(G_{p^in},\mu_{p^{i+\ord_p(n)}})
  \arrow[r, "\infl_2"]  
& H^1(G_{rp^in},\mu_{p^{i+\ord_p(n)}})
  \arrow[r, "\rest_2"]
& H^1(T,\mu_{p^{i+\ord_p(n)}})^{G_{p^in}}
\end{tikzcd}
$$
We claim that the vertical right-hand side map is injective: since $T$ acts trivially, the cohomology groups of $T$ are hom's, and this map is induced by the inclusion $\mu_{p^i}\subseteq\mu_{p^{i+\ord_p(n)}}$; it is injective because the functors $\Hom(T,-)$ and $(-)^{G_{p^in}}$ are left exact. Replacing the right-hand side terms in this diagram by the images of $\rest_1$ and $\rest_2$ respectively, yields a diagram with short exact lines, and it follows from the previous discussion that the vertical map $\img(\rest_1)\to \img(\rest_2)$ is injective. Applying the snake lemma to this new diagram, we deduce that $\coker(\rho)\hookrightarrow \coker(\theta)$. Since $\theta$ is surjective by assumption, $\rho$ is also surjective, hence $j\leq i$ since $j$ is the smallest integer satisfying this property. This proves that $\Lambda_0\mid \Lambda$, hence the result.
\end{proof}

%%%%%%%%%%%%%%%%%%%%%%%%%%%%%%%%%%%%%%%%%%%%%

\subsection{The subgroups $\Omega_p^{k,\nu}$ and their cohomology}
\label{sub:Omega}

Let $1\leq k \leq \nu$ be two integers, and let $p$ be a prime number. Let $\Omega_p^{k,\nu}$ be the subgroup of $(\Z/p^\nu\Z)^\times$ defined by
\begin{align*}
\Omega_p^{k,\nu} &:=\ker\left((\Z/p^\nu\Z)^\times \to (\Z/p^k\Z)^\times\right) \\
                 &=\{1+\beta p^k ~|~ \beta \in (\Z/p^\nu\Z)\}.
\end{align*}

We shall make extensive use of the following well-known fact: when $p$ is odd, or when $p=2$ and $k\geq 2$, the group $\Omega_p^{k,\nu}$ is cyclic of order $p^{\nu-k}$, generated by $1+p^k$.

In particular, if $p$ is odd then we have an isomorphism
$$
(\Z/p^\nu\Z)^\times \simeq \Omega_p^{1,\nu} \times (\Z/p\Z)^\times.
$$
When $p=2$ and $\nu>2$, we have an isomorphism
$$
(\Z/2^\nu\Z)^\times \simeq \Omega_2^{2,\nu} \times (\Z/4\Z)^\times.
$$

\begin{lemma}
\label{lem:2b}
Assume that $p$ is odd, or that $p=2$ and $k\geq 2$. Then
$$
H^1(\Omega_p^{k,\nu},\Z/p^\nu\Z) = 0 = H^2(\Omega_p^{k,\nu},\Z/p^\nu\Z)
$$
where $\Omega_p^{k,\nu}$ acts on $\Z/p^\nu\Z$ by multiplication.
\end{lemma}

\begin{proof}
As recalled above, under these assumptions the group $\Omega_p^{k,\nu}$ is cyclic of order $p^{\nu-k}$, generated by the element
$$
\tau_p := 1+p^k.
$$
Therefore, one has
$$
H^1(\Omega_p^{k,\nu},\Z/p^\nu\Z) = \ker(\norm)/\img(\tau_p-1)
$$
where $\norm$ is the norm morphism, obtained by summing all powers of $\tau_p$. Since $\tau_p=1+p^k$, the image of $\tau_p-1$ is just the subgroup $p^k\Z/p^\nu\Z$.

On the other hand, the norm $\norm$ is multiplication by the integer
$$
\sum_{i=0}^{p^{\nu-k}-1} (1+p^k)^i = \frac{1-(1+p^k)^{p^{\nu-k}}}{1-(1+p^k)}.
$$
But one has
$$
(1+p^k)^{p^{\nu-k}} = 1 + up^\nu
$$
for some integer $u$ coprime to $p$. Hence the norm $\norm$ is multiplication by $up^{\nu-k}$, whose kernel is $p^k\Z/p^\nu\Z$. The vanishing of the $H^1$ follows.

Turning our attention towards the $H^2$, we have
$$
H^2(\Omega_p^{k,\nu},\Z/p^\nu\Z) = (\Z/p^\nu\Z)^{\Omega_p^{k,\nu}}/\img(\norm).
$$

It follows from previous computation that the image of $\norm$ is $p^{\nu-k}\Z/p^\nu\Z$. Since the fixed points of $\Omega_p^{k,\nu}$ are the $p^k$-torsion points of $\Z/p^\nu\Z$, these two subgroups agree, hence the result.
\end{proof}

%%%%%%%%%%%%%%%%%%%%%%%%%%%%%%%%%%%%%%%%%%%%%

\subsection{The prime power obstruction}
\label{sub:primepowerobst}

According to Corollary~\ref{cor:CB_is_ab}, in order to prove Theorem~\ref{thm:main} one may replace $K$ by $K_{\ab}$.
Although we don't follow this path, the reader may assume now, until the end of the paper, that $K=K_{\ab}$ is a finite abelian extension of $\Q$.

Recall that the conductor of $K_{\ab}$ is the smallest integer $f$ such that $K_{\ab}\subseteq \Q(\zeta_f)$. We do not consider the place at infinity here.

\begin{lemma}
\label{lem:1}
Let $p\nmid \lambda$ be a prime. Then for all $n$ we have
$$
H^1(G_n,\mu_{p^{\ord_p(n)}}) = 0.
$$
\end{lemma}

\begin{proof}
We may assume that $p\mid n$. Since $\lambda$ is even, $p$ is odd and so $p-1\geq 2$. Since $p\nmid \lambda$, the extension $K(\zeta_p)/K$ is non-trivial of degree dividing $p-1$. By Galois theory, we have a surjective morphism
$$
\begin{tikzcd} 
G_n \arrow[r, twoheadrightarrow] & \Gal(K(\zeta_p)/K).
\end{tikzcd}
$$

Let $\sigma_0$ be a generator of $\Gal(K(\zeta_p)/K)$, then $\sigma_0$ is nontrivial of order coprime to $p$. Since any finite abelian group is the direct product of its $p$-part and its prime-to-$p$ part, one can lift $\sigma_0$ into an element $\sigma$ of $G_n$ which also has order coprime to $p$. Then by construction $\sigma$ does not fix $\zeta_p$, in particular
$$
(\mu_{p^{\ord_p(n)}})^{\langle\sigma\rangle} = \{1\}.
$$
Thus, the inflation-restriction sequence yields an exact sequence
\[
\begin{tikzcd}
0 
  \arrow[r] 
& H^1(G_n/\langle\sigma\rangle, (\mu_{p^{\ord_p(n)}})^{\langle\sigma\rangle}) 
  \arrow[r] 
& H^1(G_n,\mu_{p^{\ord_p(n)}}) 
  \arrow[r] 
& H^1(\langle\sigma\rangle,\mu_{p^{\ord_p(n)}})^{G_n/\langle\sigma\rangle}
\end{tikzcd}
\]
whose first term is trivial as we have seen. The third term is also trivial since $\sigma$ has order coprime to $p$. The result follows.
\end{proof}

\begin{lemma}
\label{lem:2}
Let $p\mid \lambda$. Assume that $p$ is an odd prime, or that $p=2$ and $\zeta_4\in K$.
Let $f$ be the conductor of $K_{\ab}/\Q$, and let $n$ be any integer such that $\ord_p(n)\geq \ord_p(\lambda)$.

Let $H_p:=\Gal(K(\zeta_n)/K(\zeta_{n/p^{\ord_p(n)}}))$, then the inflation map yields an isomorphism
$$
\infl : H^1(G_n/H_p, \mu_{p^k}) \simeq H^1(G_n,\mu_{p^{\ord_p(n)}})
$$
for some integer $\ord_p(\lambda) \leq k \leq \ord_p(f)$. When $n$ is a multiple of $f$, one has $k=\ord_p(f)$.
\end{lemma}

\begin{proof}
According to the proof of Corollary~\ref{cor:CB_is_ab} we may replace $K$ by $K_{\ab}$, which is a number field by assumption. To ease notation we assume, until the end of the current proof, that $K=K_{\ab}$.

We claim that
\begin{equation}
\label{eq:fixed_submodule}
(\mu_{p^{\ord_p(n)}})^{H_p} = \mu_{p^k}
\end{equation}
for some $k\leq \ord_p(f)$. The fixed field of $H_p$ is $K(\zeta_{n/p^{\ord_p(n)}})$, which is an extension of $K$ unramified above $p$. In particular, the conductor of $K(\zeta_{n/p^{\ord_p(n)}})/\Q$ has the same $p$-adic valuation as the conductor $f$ of $K/\Q$. It follows that the field $K(\zeta_{n/p^{\ord_p(n)}})$ cannot contain $p^k$th roots of unity with $k>\ord_p(f)$, hence the claim.

Next we claim that $k=\ord_p(f)$ when $n$ is a multiple of $f$. Let us write
$$
f=p^{\ord_p(f)}m
$$
with $m$ coprime to $p$. Then we have
\begin{equation}
\label{eq:cyclotomic_agreement}
K(\zeta_m) = \Q(\zeta_f).
\end{equation}

Indeed, $K\subseteq \Q(\zeta_f)$ by definition of the conductor, and since $p\mid \lambda$ we have
$$
\Q(\zeta_{p^{\ord_p(\lambda)}m})\subseteq K(\zeta_m)\subseteq \Q(\zeta_f).
$$

It follows from Galois theory that $\Gal(\Q(\zeta_f)/K(\zeta_m))$ is a subgroup of
$$
\Gal\left(\Q(\zeta_f)/\Q(\zeta_{p^{\ord_p(\lambda)}m})\right)=\Omega_p^{\ord_p(\lambda),\ord_p(f)}\times\{1\}\leq (\Z/p^{\ord_p(f)}\Z)^\times\times(\Z/m\Z)^\times,
$$
which is a cyclic $p$-group, because $p$ is odd, or $p=2$ and $\zeta_4\in K$. If $\Gal(\Q(\zeta_f)/K(\zeta_m))$ were non-trivial, it would contain the unique subgroup of order $p$ of this cyclic group, namely $\Omega_p^{\ord_p(f)-1,\ord_p(f)}\times\{1\}=\Gal(\Q(\zeta_f)/\Q(\zeta_{f/p}))$, and then $K\subseteq K(\zeta_m)\subseteq\Q(\zeta_{f/p})$, contradicting the definition of $f$. Hence $\Gal(\Q(\zeta_f)/K(\zeta_m))$ is trivial, which proves \eqref{eq:cyclotomic_agreement}.

A consequence of \eqref{eq:cyclotomic_agreement} is that $\zeta_{p^{\ord_p(f)}}$ belongs to $\Q(\zeta_f)=K(\zeta_m)$, and it follows from the previous discussion that $k=\ord_p(f)$ when one takes $n=f$. More generally, if $n$ is a multiple of $f$ then one can write $n=p^{\ord_p(n)}N$ with $\ord_p(n)\geq \ord_p(f)$ and $m\mid N$; then $\Q(\zeta_f)$ is a subfield of $K(\zeta_N)$, hence we also have that $k=\ord_p(f)$ in this case.

In view of \eqref{eq:fixed_submodule}, the inflation-restriction sequence yields an exact sequence
\[
\begin{tikzcd}
0 
  \arrow[r] 
& H^1(G_n/H_p, \mu_{p^k}) 
  \arrow[r] 
& H^1(G_n,\mu_{p^{\ord_p(n)}}) 
  \arrow[r] 
& H^1(H_p,\mu_{p^{\ord_p(n)}})^{G_n/H_p}
\end{tikzcd}
\]
We shall prove that the third term is trivial. We have
\begin{align*}
H_p &\simeq \Gal\left(\Q(\zeta_n)/\Q(\zeta_n)\cap K(\zeta_{n/p^{\ord_p(n)}})\right) \\
    &\subseteq \Gal\left(\Q(\zeta_n)/\Q(\zeta_{n/p^{\ord_p(n)-\ord_p(\lambda)}})\right) \\
    &\simeq\Gal(\Q(\zeta_{p^{\ord_p(n)}})/\Q(\zeta_{p^{\ord_p(\lambda)}})) 
\end{align*}
where the inclusion follows from the fact that $\zeta_{p^{\ord_p(\lambda)}}$ belongs to $K$. By Galois theory, we have
$$
\Gal\left(\Q(\zeta_{p^{\ord_p(n)}})/\Q(\zeta_{p^{\ord_p(\lambda)}})\right) = \Omega_p^{\ord_p(\lambda),\ord_p(n)}
$$
with its natural action on $\mu_{p^{\ord_p(n)}}$ by exponentiation.

If $p$ is odd, or if $p=2$ and $\zeta_4\in K$, this group $\Omega_p^{\ord_p(\lambda),\ord_p(n)}$ is cyclic of order $p^{\ord_p(n)-\ord_p(\lambda)}$, and therefore has, for each $\ell\geq \ord_p(\lambda)$, a unique subgroup of order $p^{\ord_p(n)-\ell}$ which is none other than $\Omega_p^{\ell,\ord_p(n)}$
(see \S{}\ref{sub:Omega}). It follows that $H_p=\Omega_p^{\ell,\ord_p(n)}$ for some $\ell\geq \ord_p(\lambda)$, hence $H^1(H_p,\mu_{p^{\ord_p(n)}})$ is trivial by Lemma~\ref{lem:2b}.
We note for the record that
$$
\Q(\zeta_n)\cap K(\zeta_{n/p^{\ord_p(n)}}) = \Q(\zeta_{n/p^{\ord_p(n)-\ell}})
$$
hence $\ell=k$ where $k$ is the integer defined previously.
\end{proof}

\begin{lemma}
\label{lem:3}
Let $p=2$ and assume that $\zeta_4\notin K$.
Let $f$ be the conductor of $K_{\ab}/\Q$.
Let $n$ be an integer with $\ord_2(n)\geq 2$, and let $H_4:=\Gal(K(\zeta_n)/K(\zeta_{n/2^{\ord_2(n)}},\zeta_4))$. Then the inflation map yields an isomorphism
$$
\infl : H^1(G_n/H_4, \mu_{2^k}) \simeq H^1(G_n,\mu_{2^{\ord_2(n)}})
$$
for some integer $2\leq k \leq \max(2,\ord_2(f))$. When $n$ is a multiple of $f$, one has $k=\max(2,\ord_2(f))$.
\end{lemma}

\begin{proof}
As in the proof of Lemma~\ref{lem:2}, assume that $K=K_{\ab}$. Let $K':=K(\zeta_4)$; it is an abelian field containing $\zeta_4$, whose conductor is $\lcm(4,f)$, hence has $2$-adic valuation $\max(2,\ord_2(f))$. Since $4\mid n$, we have $K'(\zeta_n)=K(\zeta_n)$, and $H_4$ is the group $H_2$ of Lemma~\ref{lem:2} for the field $K'$. By the proof of Lemma~\ref{lem:2} applied to $K'$ and with the same notation, $H_4$ is isomorphic to a subgroup of $\Omega_2^{2,\ord_2(n)}$, hence is of the form $\Omega_2^{k,\ord_2(n)}$ for some $2\leq k \leq \max(2,\ord_2(f))$.
The result then follows from the inflation-restriction exact sequence, combined with the vanishing of $H^1(\Omega_2^{k,\ord_2(n)},\Z/2^{\ord_2(n)}\Z)$ granted by Lemma~\ref{lem:2b}.
\end{proof}

\begin{lemma}
\label{lem:borne_inf}
Let $p\mid \lambda$. Assume that $p$ is odd, or that $p=2$ and $4\mid f$. Then we have
$$
p^{\ord_p(f)-\ord_p(\lambda)} \mid \Lambda.
$$
When $\zeta_4\notin K$, we also have that $4 \mid \Lambda$.
\end{lemma}

Note that $\ord_p(f)\geq \ord_p(\lambda)$, except in the special case when $p=2$ and $f$ is odd, which is covered by the last sentence of the statement.

\begin{proof}
Assume that $K=K_{\ab}$ as previously. Let $f=p^{\ord_p(f)}m$ with $p\nmid m$.
\begin{itemize}
\item Assume $p$ is odd, or that $p=2$ and $\zeta_4\in K$. In this case, according to \eqref{eq:cyclotomic_agreement} in the proof of Lemma~\ref{lem:2}, we have
$$
\text{for all } l\leq \ord_p(f), \qquad K(\zeta_m) = K(\zeta_{p^lm}) = K(\zeta_f).
$$

In particular, $\zeta_{p^{\ord_p(\lambda)}}$ is a $p^{\ord_p(f)-\ord_p(\lambda)}$th power in $K(\zeta_{p^{\ord_p(f)-\ord_p(\lambda)}m})$, hence is also a $p^{\ord_p(f)-\ord_p(\lambda)}m$th power in this field, but is not a $p$th power in $K$. The result follows by definition of the Chevalley-Bass number of $K$.
\item Assume $p=2$, $4\mid f$ and $\zeta_4\notin K$. Then the abelian field $K(\zeta_4)$ has conductor $f$ and contains $\zeta_4$, hence \eqref{eq:cyclotomic_agreement} applied to $K(\zeta_4)$ gives
$$
\text{for all } 2\leq l\leq \ord_2(f), \qquad K(\zeta_{4m})= K(\zeta_{2^lm}) = \Q(\zeta_f).
$$
The same reasoning as above proves the result provided $\ord_2(f)\geq 3$. If $\ord_2(f)=2$ then the statement follows from the next case, which proves that $4\mid \Lambda$ whenever $\zeta_4\notin K$.
\item If $\zeta_4\notin K$ then $-1$ is not a square in $K$, hence $-4=(1+\zeta_4)^4$ is a $4$th power in $K(\zeta_4)$ but is not a square in $K$, so the Chevalley-Bass number of $K$ is divisible by $4$.
\end{itemize}
\end{proof}

%%%%%%%%%%%%%%%%%%%%%%%%%%%%%%%%%%%%%%%%%%%%%

\subsection{Proof of the main results}

We now collect the results of our previous computation.

\begin{lemma}
\label{lem:key}
The Chevalley-Bass number of $K$ divides the integer
$$
\ff:=
\begin{cases}
\prod_{p\mid \lambda} p^{\ord_p(f)} & \text{if $4\mid f$} \\
4\prod_{p\mid \lambda} p^{\ord_p(f)} & \text{if $f$ is odd}.
\end{cases}
$$
\end{lemma}

\begin{proof}
According to Corollary~\ref{cor:prime_by_prime} (more precisely, the first half of its proof) it suffices to prove that, for every integer $n>0$ and every prime number $p$ dividing $\ff{}n$, the map
\begin{equation}
\label{eq:foo_map}
H^1(G_{\ff{}n}, \mu_{p^{\ord_p(\ff)}}) \to H^1(G_{\ff{}n}, \mu_{p^{\ord_p(\ff{}n)}})
\end{equation}
is surjective.
According to Lemma~\ref{lem:1}, if $p\nmid \lambda$ then $H^1(G_{\ff{}n},\mu_{p^{\ord_p(\ff{}n)}})=0$, hence in this case the result is trivially true. Let us now assume that $p\mid \lambda$, then $\ord_p(\ff)=\ord_p(f)$ if $p$ is odd, and $\ord_2(\ff)=\max(2,\ord_2(f))$ by definition of $\ff$. When $p$ is odd, or when $p=2$ and $\zeta_4\in K$, it follows from Lemma~\ref{lem:2} that there exists a subgroup $H_p\leq G_{\ff{}n}$ for which the inflation map
$$
\infl:H^1(G_{\ff{}n}/H_p, \mu_{p^k}) \to H^1(G_{\ff{}n}, \mu_{p^{\ord_p(\ff{}n)}})
$$
is an isomorphism, for some integer $k\leq \ord_p(f)=\ord_p(\ff)$. Then $(\mu_{p^{\ord_p(\ff{}n)}})^{H_p} = (\mu_{p^{\ord_p(\ff)}})^{H_p} =\mu_{p^k}$ hence the inflation map above factors as
$$
\begin{tikzcd}
H^1(G_{\ff{}n}/H_p, \mu_{p^k}) \arrow[r, "\infl'"] & H^1(G_{\ff{}n}, \mu_{p^{\ord_p(\ff)}}) \arrow[r] & H^1(G_{\ff{}n}, \mu_{p^{\ord_p(\ff{}n)}}),
\end{tikzcd}
$$
where $\infl'$ is the inflation map for $\mu_{p^{\ord_p(\ff)}}$, and the right-hand side map is \eqref{eq:foo_map}.
The bijectivity of the composition of the two maps implies the surjectivity of the right-hand side map, hence the result.
In the case when $p=2$ and $\zeta_4\notin K$, the proof can be done along the same lines, using Lemma~\ref{lem:3} instead.
\end{proof}

\begin{lemma}
\label{lem:exponent}
For every integer $n$, the group $H^1(G_n, \mu_n)$ is killed by $\lambda$. Moreover, the group
$H^1(G_{\ff{}n},\mu_{\ff{}n})$ has exponent $\lambda$ for infinitely many values of $n$ (with notation of Lemma~\ref{lem:key}), and its order is unbounded when $n$ varies.
\end{lemma}

\begin{proof}
Recall that we have an isomorphism
$$
H^1(G_n, \mu_n) \simeq \prod_{\text{primes}~p\mid n} H^1(G_n, \mu_{p^{\ord_p(n)}}).
$$
Therefore, by Lemma~\ref{lem:1}, in order to prove the first statement it suffices to prove that, for all primes $p$ dividing $\lambda$, the group $H^1(G_n,\mu_{p^{\ord_p(n)}})$ is killed by $p^{\ord_p(\lambda)}$. This trivially holds when $\ord_p(n)<\ord_p(\lambda)$, so we may assume that $\ord_p(n)\geq \ord_p(\lambda)$.

Assume first that $p$ is odd, or that $p=2$ and $\zeta_4\in K$. We claim that
\begin{equation}
\label{eq:omegaquotient}
\begin{split}
G_n\twoheadrightarrow \Gal(K(\zeta_{p^{\ord_p(n)}})/K)  &\simeq \Gal\left(\Q(\zeta_{p^{\ord_p(n)}})/\Q(\zeta_{p^{\ord_p(n)}})\cap K\right) \\
    &= \Gal(\Q(\zeta_{p^{\ord_p(n)}})/\Q(\zeta_{p^{\ord_p(\lambda)}})) = \Omega_p^{\ord_p(\lambda),\ord_p(n)}
\end{split}
\end{equation}

The first equality requires a careful check: since $\zeta_{p^{\ord_p(\lambda)}}$ belongs to $K$, the second Galois group is a subgroup of the last. We may apply the same argument in the proof of Lemma~\ref{lem:2}: since $\Gal(\Q(\zeta_{p^{\ord_p(n)}})/\Q(\zeta_{p^{\ord_p(\lambda)}}))=\Omega_p^{\ord_p(\lambda),\ord_p(n)}$
is cyclic of order $p^{\ord_p(n)-\ord_p(\lambda)}$, its subgroups are of the form $\Gal(\Q(\zeta_{p^{\ord_p(n)}})/\Q(\zeta_{p^\ell}))$, hence $\Q(\zeta_{p^{\ord_p(n)}})\cap K=\Q(\zeta_{p^\ell})$ for some $\ell\geq \ord_p(\lambda)$, and we conclude that $\ell=\ord_p(\lambda)$ by definition of $\lambda$, hence the claim.

It follows from \eqref{eq:omegaquotient} that $\Omega_p^{\ord_p(\lambda),\ord_p(n)}$ is a quotient of $G_n$. Let $\tau=1+p^{\ord_p(\lambda)}$, which is a generator of $\Omega_p^{\ord_p(\lambda),\ord_p(n)}$, and let $\tilde{\tau} \in G_n$ be a lift of $\tau$. 
Then by construction $\tilde{\tau}$ also acts by multiplication by $1+p^{\ord_p(\lambda)}$ on $\mu_{p^{\ord_p(n)}}$. But Sah's Lemma \cite[Lemma~A.2]{BakerRibet2003} implies that the group $H^1(G_n,\mu_{p^{\ord_p(n)}})$ is killed by $\tilde{\tau}-1$, hence the result.

It remains to treat the case when $p=2$ and $\zeta_4\notin K$. We may assume that $\ord_2(n)\geq 2$, in which case the map
$$
G_n \to (\Z/4\Z)^\times
$$
is surjective, since the action of $G_n$ on the $4$th roots of unity is not trivial. In particular, there exists an element $\sigma\in G_n$ such that $\sigma\equiv 3 \pmod{4}$. Let $\sigma=3+4\beta$ for some $\beta$, then $\sigma-1$ kills the group $H^1(G_n,\mu_{2^{\ord_2(n)}})$. But $\sigma-1 = 2\cdot (1+2\beta)$ and since $(1+2\beta)$ is odd it is an automorphism of this group. Therefore, the group is killed by $2$, hence the result.

Let us prove the second statement. Consider the short exact sequence of $G_{\ff{}n}$-modules
$$
1\to \mu_{\lambda} \to \mu_{\ff{}n}\to \mu_{\ff{}n/\lambda}\to 1
$$

Since $(\mu_{\ff{}n/\lambda})^{G_{\ff{}n}}$ is a subgroup of $\mu_{\lambda}$, it is cyclic of order $\lambda'\mid \lambda$. Therefore, we have an exact sequence
$$
1\to \mu_{\lambda'} \to H^1(G_{\ff{}n},\mu_{\lambda}) \to H^1(G_{\ff{}n},\mu_{\ff{}n}).
$$

In view of this exact sequence it is enough to show that, for suitable values of $n$, $H^1(G_{\ff{}n},\mu_{\lambda})$ contains a subgroup isomorphic to $(\Z/\lambda\Z)^r$ where $r$ can be made arbitrarily large.

Since $G_{\ff{}n}$ acts trivially on $\mu_{\lambda}$, we have
\begin{equation}
\label{eq:Hom}
H^1(G_{\ff{}n},\mu_{\lambda}) = \Hom(G_{\ff{}n},\mu_{\lambda}).
\end{equation}

On the other hand, if $n$ is coprime to $2f$ we have
$$
(\Z/n\Z)^\times \simeq \Gal\left(K(\zeta_{\ff{}n})/K(\zeta_{\ff})\right) \subseteq G_{\ff{}n}.
$$

Let us choose, for each prime $p$ dividing $\lambda$, a prime $q$ such that $p^{\ord_p(\lambda)}$ divides $q-1$ (for which there exist infinitely many choices), and let $n$ be the product of all these $q$. Then $(\Z/n\Z)^\times$ has a cyclic subgroup of order $p^{\ord_p(\lambda)}$ for each such $p$, hence by the Chinese remainder Theorem it has a cyclic subgroup of order $\lambda$. Since all the groups involved are abelian, it follows that $G_{\ff{}n}$ has a cyclic quotient of order $\lambda$. So there exists a surjective map $G_{\ff{}n}\to \mu_{\lambda}$ which corresponds by \eqref{eq:Hom} to an element of order exactly $\lambda$ in the group $H^1(G_{\ff{}n},\mu_{\lambda})$. In order to construct $r$ independent elements of order $\lambda$ in this group, it suffices to pick, for each prime $p\mid\lambda$, $r$ values of $q$ such that $p^{\ord_p(\lambda)}$ divides $q-1$.
\end{proof}

\begin{proof}[Proof of Theorem~\ref{thm:main}]
According to Lemma~\ref{lem:key}, the Chevalley-Bass number $\Lambda$ of $K$ divides $\ff$. It follows that, for every $n>0$, $\Lambda$ divides $\ff{}n$, so by Lemma~\ref{lem:cohomologicalCB}(ii) the group $H^1(G_{\ff{}n},\mu_{\ff{}n})$ is a quotient of $H^1(G_{\ff{}n},\mu_{\Lambda})$, hence is killed by $\Lambda$. But by Lemma~\ref{lem:exponent} this group $H^1(G_{\ff{}n},\mu_{\ff{}n})$ has exponent $\lambda$ for suitable $n$, hence $\lambda$ divides $\Lambda$.
It follows from Lemma~\ref{lem:borne_inf} that $\Lambda$ is divisible by $\ff/\lambda$ when $4\mid f$, and by the odd number $\ff/2\lambda$ when $f$ is odd. Finally, $\Lambda$ is always divisible by $4$: when $\zeta_4\in K$ this holds because $4\mid\lambda\mid \Lambda$, and when $\zeta_4\notin K$ this is the last sentence of Lemma~\ref{lem:borne_inf}. The result follows.
\end{proof}

\begin{proof}[Proof of Corollary~\ref{cor:intro}]
Let $x\in K^\times$ which is a $\lambda{}n$th power in $K(\zeta_{\lambda{}n})^\times$. Then, by exactness of the inflation-restriction sequence
\[
\begin{tikzcd}
0 
  \arrow[r] 
& H^1(G_{\lambda{}n},\mu_{\lambda{}n})
  \arrow[r] 
& K^\times/(K^\times)^{\lambda{}n}
  \arrow[r] 
& K(\zeta_{\lambda{}n})^\times/(K(\zeta_{\lambda{}n})^\times)^{\lambda{}n}
\end{tikzcd}
\]
the class of $x$ in $K^\times/(K^\times)^{\lambda{}n}$ belongs to the image of the inflation map, which is killed by $\lambda$ according to Lemma~\ref{lem:exponent}. In particular, the class of $x^\lambda$ is trivial in $K^\times/(K^\times)^{\lambda{}n}$. This means that there exists $y\in K^\times$ such that $x^\lambda = y^{\lambda{}n}$, and it follows immediately that $x=\varepsilon y^n$ for some $\lambda$th root of unity $\varepsilon\in K$.
\end{proof}

\begin{proof}[Proof of Corollary~\ref{cor:LaurentBilu}]
Pick a positive integer $n$ such that $\alpha_1^n,\dots,\alpha_s^n$ belong to $K$. The first step in the proof of Proposition~4.4 of \cite{BLNOW2025} proves that the quantities $\alpha_1,\dots,\alpha_s$ belong to $K(\zeta_n)$. Then $\alpha_1^{\lambda{}n},\dots,\alpha_s^{\lambda{}n}$ are elements of $K$ which are $\lambda{}n$th powers in $K(\zeta_n)\subset K(\zeta_{\lambda{}n})$, hence by Corollary~\ref{cor:intro} there exist $y_i \in K$ and $\lambda$th roots of unity $\varepsilon_i$ such that 
$$
\alpha_i^{\lambda{}n} = \varepsilon_i y_i^n \qquad \text{for $i=1,\dots,s$.}
$$
The result follows.
\end{proof}

Let us add a few comments on Remark~\ref{remark:forYuri}: one can replace $\Lambda$ by $\lambda$ in Proposition~4.5 of \cite{BLNOW2025} since this statement is merely a consequence of Proposition~4.4 of \emph{loc. cit.}

A last comment about the improvement of Proposition~4.7 of \cite{BLNOW2025}: let $K$ be the splitting field of a polynomial of degree $m\geq 2$ with coefficients in $\Q$. Then as recalled in the proof of Proposition~6.5 of \cite{bilu2023} we have the estimate $[K_{\ab}:\Q]\leq 3^{m/3}$, which follows from results of Kov\'acs--Praeger \cite{KovacsPraeger1989} and Aschbacher--Guralnick \cite{AschbacherGuralnick1989} on abelian quotients of permutation groups.  The number of roots of unity contained in $K$ is the largest integer $\lambda$ such that $\Q(\zeta_\lambda)\subseteq K_{\ab}$. In particular, it satisfies $\varphi(\lambda)\leq [K_{\ab}:\Q]\leq 3^{m/3}$, where $\varphi$ denotes Euler's totient function. Since $\lambda/2\leq \varphi(\lambda)^2$ holds for all $\lambda$, it follows that $\lambda/2\leq 3^{2m/3} \leq \exp(m)$, and the result follows from the improved version of Proposition~4.5.

We conclude this section with an explicit upper bound for the Chevalley-Bass number in terms of the degree of $K_{\ab}$.

\begin{corollary}
\label{cor:degree_bound}
Let $d:=[K_{\ab}:\Q]$. Then
\begin{equation}
\label{eq:degree_upperbound}
\Lambda \mid 2d\frac{\lambda}{\varphi(\lambda)} \qquad \text{and}\qquad \Lambda \leq 2d\left(2+\log_2d\right).
\end{equation}
In particular, if $K$ is the splitting field of a polynomial of degree $m\geq 2$ with rational coefficients, then
$$
\Lambda\leq 2\cdot 3^{m/3}\left(2+\frac{m}{3}\log_23\right).
$$
\end{corollary}

\begin{proof}
By Corollary~\ref{cor:CB_is_ab}, we may assume that $K=K_{\ab}$. Let $\rad(\lambda)$ be the product of primes dividing $\lambda$. We claim that, for all primes $p\mid\lambda$,
\begin{equation}
\label{eq:compare_f_and_d}
\ord_p(f) \leq
\begin{cases}
\ord_p\left(\frac{d}{\varphi(\rad(\lambda))}\right)+1 & \text{if $p$ is odd} \\
\ord_2\left(\frac{d}{\varphi(\rad(\lambda))}\right)+2 & \text{if $p=2$}.
\end{cases}
\end{equation}
Write $f=p^{\ord_p(f)}N$ with $p\nmid N$. By Galois theory, $G_f=\Gal(\Q(\zeta_f)/K)$ is a subgroup of $(\Z/f\Z)^\times = (\Z/p^{\ord_p(f)}\Z)^\times \times (\Z/N\Z)^\times$. If $p$ is odd, the group $\Omega_p^{1,\ord_p(f)}\times\{1\}$ is cyclic of order $p^{\ord_p(f)-1}$ and intersects $G_f$ trivially: if it were not the case, then $G_f$ would contain $\Omega_p^{\ord_p(f)-1,\ord_p(f)}\times\{1\}$ hence $K$ would be a subfield of $\Q(\zeta_{f/p})$ which contradicts the definition of $f$. Therefore, $\Omega_p^{1,\ord_p(f)}$ injects into the quotient $\Gal(\Q(\zeta_f)/\Q)/G_f = \Gal(K/\Q)$. In addition, since $\Omega_p^{1,\ord_p(f)}$ acts trivially on $p$th roots of unity, and on $q$th roots of unity for all $q$ prime to $p$, it acts trivially on $\Q(\zeta_{\rad(\lambda)})$, hence its image lies in the subgroup $\Gal(K/\Q(\zeta_{\rad(\lambda)}))$. Therefore,
$$
p^{\ord_p(f)-1}\mid [K:\Q(\zeta_{\rad(\lambda)})] = \frac{d}{\varphi(\rad(\lambda))}
$$
which proves the first case of \eqref{eq:compare_f_and_d}.

If $p=2$ and $4\mid f$ then $\Omega_2^{2,\ord_2(f)}$ is cyclic of order $2^{\ord_2(f)-2}$ and has the same property, which proves the second case of \eqref{eq:compare_f_and_d}.

Combining this with Lemma~\ref{lem:key} we obtain
$$
\Lambda\mid \ff \mid 2^{\ord_2\left(\frac{d}{\varphi(\rad(\lambda))}\right)+2}\prod_{\substack{p\mid\lambda \\ p~\text{odd}}}p^{\ord_p\left(\frac{d}{\varphi(\rad(\lambda))}\right)+1} \mid \frac{d}{\varphi(\rad(\lambda))}\cdot 2\rad(\lambda).
$$
On the other hand, we observe that
$$
\frac{\rad(\lambda)}{\varphi(\rad(\lambda))} = \prod_{p\mid\lambda} \frac{p}{p-1} = \frac{\lambda}{\varphi(\lambda)} 
$$
which concludes the proof of the divisibility statement in \eqref{eq:degree_upperbound}.

Let $p_1<\dots<p_k$ be the prime factors of $\lambda$, then the obvious bound $p_i\geq i+1$ yields
\begin{equation}
\label{eq:lambda_over_phi_lambda}
\frac{\lambda}{\varphi(\lambda)} =\prod_{i=1}^k \left(1+ \frac{1}{p_i-1}\right) \leq\prod_{i=1}^k \left(1+\frac{1}{i}\right) =k+1.
\end{equation}
Moreover $\prod_{i=1}^k(p_i-1)$ divides $\varphi(\lambda)$, which divides $d$, and each odd prime contributes a factor at least $2$, so the number $k$ of prime factors of $\lambda$ is at most $1+\log_2d$. Combining this with \eqref{eq:lambda_over_phi_lambda} we conclude that $\lambda/\varphi(\lambda)$ is bounded above by $2+\log_2d$, and the second statement of \eqref{eq:degree_upperbound} follows. The last assertion follows from the bound $d\leq 3^{m/3}$, for which references are given before the statement of the Corollary.
\end{proof}

%%%%%%%%%%%%%%%%%%%%%%%%%%%%%%%%%%%%%%%%%%%%%

\subsection{Computation of the Chevalley-Bass number}
\label{sub:algorithm}

In this section, we give an algorithm which allows one to compute the Chevalley-Bass number of a field $K$. It relies on the following Lemma, which allows one to reduce to a finite number of steps in order to compute the Chevalley-Bass number $\Lambda$ of $K$.

\begin{lemma}
\label{lem:CB_algorithm}
Let $p\mid \lambda$. Assume that $p$ is an odd prime, or that $p=2$ and $\zeta_4\in K$.
Then $\ord_p(\Lambda)$ is the smallest integer
$$
\max\{\ord_p(\lambda),\ord_p(f)-\ord_p(\lambda)\} \leq j \leq \ord_p(f)
$$
such that, for all $n$ of the form
$$
n=p^t\prod_{\substack{\text{primes}~q\mid f \\ q\equiv 1 \bmod p }} q, \quad \text{with } j \leq t \leq \ord_p(f)
$$
the map $H^1(G_n,\mu_{p^j}) \to H^1(G_n,\mu_{p^t})$ is surjective.

When $p=2$ and $\zeta_4\notin K$, the same holds with $\max\{2,\ord_2(f)-1\}\leq j \leq \max(2,\ord_2(f))$, the condition on $t$ being replaced by $j \leq t\leq \max(2,\ord_2(f))$.
\end{lemma}

When $p=2$ and $f$ is odd, the statement above implies that $\ord_2(\Lambda)=2$, which also follows from Theorem~\ref{thm:main}.

Before we start the proof, let us recall an elementary consequence of Galois theory: given two integers $n$ and $n'$ with $n'\mid n$, we have a commutative diagram with exact lines and vertical inclusions
\begin{equation}
\label{eq:GntoGn'}
\begin{tikzcd}
1 \arrow[r] & \mathrm{Gal}(K(\zeta_n)/K(\zeta_{n'})) \arrow[r] \arrow[d, hook]
            & G_n \arrow[r] \arrow[d, hook]
            & G_{n'} \arrow[r] \arrow[d, hook]
            & 1 \\
1 \arrow[r] & \ker(\mathrm{red}^\times) \arrow[r]
            & (\Z/n\Z)^\times \arrow[r, "\mathrm{red}^\times"]
            & (\Z/n'\Z)^\times \arrow[r]
            & 1
\end{tikzcd}
\end{equation}
If in addition $\Q(\zeta_n)\cap K=\Q(\zeta_{n'})\cap K$, then $G_n$ is the inverse image of $G_{n'}$ by the reduction map $\red^\times$, in other terms the vertical inclusion on the left is an equality.

\begin{proof}
By Theorem~\ref{thm:main}, we have
\begin{align}
&\max\{\ord_p(\lambda),\ord_p(f)-\ord_p(\lambda)\} \leq \ord_p(\Lambda) \leq \ord_p(f) && \text{if $p$ is odd, or $p=2$ and $\zeta_4\in K$} \label{eq:case1} \\
&\max\{2,\ord_2(f)-1\} \leq \ord_2(\Lambda) \leq \max(2,\ord_2(f)) && \text{if $p=2$ and $\zeta_4\notin K$.} \label{eq:case2}
\end{align}

By Corollary~\ref{cor:prime_by_prime}, we know that  $\ord_p(\Lambda)$ is the smallest integer $j$ such that, for all $n$ such that $\ord_p(n)\geq j$, the map
\begin{equation}
\label{eq:be_surjective}
H^1(G_n,\mu_{p^j}) \to H^1(G_n,\mu_{p^{\ord_p(n)}})
\end{equation}
is surjective. Starting from any integer $n$, we shall show that, in order to check the surjectivity of this map, one may replace $n$ by some integer described in the Lemma. For ease of notation, we let $t:=\ord_p(n)$ throughout the proof.
\begin{itemize}
\item One can assume that $t\leq\ord_p(f)$ in case \eqref{eq:case1}, and that $t\leq \ord_2(\ff)=\max(2,\ord_2(f))$ in case \eqref{eq:case2}. Indeed, assume that $t>\ord_p(\ff)$, then one can write $n=p^tN$ with $p\nmid N$. We claim that one can replace $n$ by $n':=p^{\ord_p(\ff)}N$.

Since $\gcd(n,f)=\gcd(n',f)$ we have that $\Q(\zeta_n)\cap K=\Q(\zeta_{n'})\cap K$. It follows from \eqref{eq:GntoGn'} that the kernel $U$ of $G_n\to G_{n'}$ is the full kernel of reduction, which is isomorphic to $\Omega_p^{\ord_p(\ff),t}$ with its natural action on $\mu_{p^t}$. It follows from Lemma~\ref{lem:2b} that $H^1(U,\mu_{p^t})=0$. Since $(\mu_{p^t})^U=\mu_{p^{\ord_p(\ff)}}$, the inflation-restriction exact sequence yields an isomorphism
$$
\infl : H^1(G_{n'},\mu_{p^{\ord_p(\ff)}}) \to H^1(G_n,\mu_{p^t}).
$$
Consider the following commutative square, whose vertical maps are inflations (note that $U$ acts trivially on $\mu_{p^j}$)
\[
\begin{tikzcd}
H^1(G_{n'},\mu_{p^j}) \arrow[r] \arrow[d,"\infl"] & H^1(G_{n'},\mu_{p^{\ord_p(\ff)}}) \arrow[d,"\infl"] \\
H^1(G_n,\mu_{p^j}) \arrow[r] & H^1(G_n,\mu_{p^t})
\end{tikzcd}
\]
The vertical map on the right is an isomorphism. So, the surjectivity of the top map implies the surjectivity of the bottom map. This completes the first reduction step.

\item One can assume that the prime factors of $n$ divide $f$. Indeed, write $n=n'N$ with $N$ coprime to $f$, then by the Chinese remainder theorem we have
$$
G_n = G_{n'} \times (\Z/N\Z)^{\times}
$$
where $(\Z/N\Z)^{\times}$ acts trivially on $\mu_{p^{\ord_p(n)}}$. Then the inflation-restriction exact sequence for $G_{n'} \times (\Z/N\Z)^{\times}$ degenerates \cite{Jannsen1990}, hence we have a commutative diagram with exact lines
$$
\begin{tikzcd}
1 \arrow[r] & H^1((\Z/N\Z)^{\times}, (\mu_{p^j})^{G_{n'}}) \arrow[r] \arrow[d]
            & H^1(G_n, \mu_{p^j}) \arrow[r] \arrow[d]
            & H^1(G_{n'}, \mu_{p^j}) \arrow[r] \arrow[d]
            & 1 \\
1 \arrow[r] & H^1((\Z/N\Z)^{\times}, (\mu_{p^t})^{G_{n'}}) \arrow[r]
            & H^1(G_n, \mu_{p^t}) \arrow[r]
            & H^1(G_{n'}, \mu_{p^t}) \arrow[r]
            & 1
\end{tikzcd}
$$
Since $t\geq j\geq \ord_p(\lambda)$ by assumption, we have that $(\mu_{p^j})^{G_{n'}}=(\mu_{p^t})^{G_{n'}}=\mu_{p^{\ord_p(\lambda)}}$, hence the vertical map on the left is an equality. It follows from the snake lemma that the vertical map in the middle is surjective if and only if the vertical map on the right is.

\item Assuming the prime factors of $n$ divide $f$, one can replace $n$ by the integer
$$
n':=p^t\prod_{\substack{\text{primes}~q\mid n \\ q\equiv 1 \bmod p}} q \qquad (\text{where } t=\ord_p(n)).
$$

Firstly, observe that the kernel of the reduction (modulo $n'$) map
$$
\red^\times : (\Z/n\Z)^\times \to (\Z/p^t\Z)^\times \times \prod_{\substack{\text{primes}~q\mid n \\ q\equiv 1 \bmod p}} (\Z/q\Z)^\times
$$
is a product of groups of the form $\Omega_q^{1,r}$ (when $q\equiv 1 \bmod p$) and of the form $(\Z/q^r\Z)^\times$ otherwise, hence is coprime to $p$. Then by \eqref{eq:GntoGn'} we have an exact sequence $1\to H \to G_{n} \to G_{n'} \to 1$ for some group $H$ whose order is coprime to $p$, and which acts trivially on $\mu_{p^t}$. Then the inflation-restriction exact sequence yields an isomorphism
$$
\infl: H^1(G_{n'},\mu_{p^t}) \simeq H^1(G_n,\mu_{p^t})
$$
and similarly for $\mu_{p^j}$. Thus, the surjectivity of $H^1(G_n,\mu_{p^j}) \to H^1(G_n,\mu_{p^t})$ is equivalent to the surjectivity of $H^1(G_{n'},\mu_{p^j}) \to H^1(G_{n'},\mu_{p^t})$.

\item One can replace the integer $n'$ in the previous step by
$$
n'':=p^t\prod_{\substack{\text{primes}~q\mid f \\ q\equiv 1 \bmod p }} q.
$$
Indeed, $n''=rn'$ with $r$ coprime to $p$, so that $t=\ord_p(n')=\ord_p(n'')$ and $\ker(G_{n''}\to G_{n'})$ acts trivially on $\mu_{p^t}$. The argument in the proof of Corollary~\ref{cor:prime_by_prime} (inflation-restriction and the snake lemma) shows that the surjectivity of $H^1(G_{n''},\mu_{p^j})\to H^1(G_{n''},\mu_{p^t})$ implies that of $H^1(G_{n'},\mu_{p^j})\to H^1(G_{n'},\mu_{p^t})$.
\end{itemize}
\end{proof}

\subsubsection*{An algorithm to compute the Chevalley-Bass number}

If $K$ is a number field, there exists an algorithm which computes $K_{\ab}$: one first computes the Galois closure $\tilde{K}$ of $K/\Q$ and its Galois group, then $K_{\ab}$ is the field fixed by the group generated by $\Gal(\tilde{K}/K)$ and by the derived subgroup of $\Gal(\tilde{K}/\Q)$.

Assume that $K_{\ab}/\Q$ is known. According to Corollary~\ref{cor:CB_is_ab}, the fields $K$ and $K_{\ab}$ have the same Chevalley-Bass number, so we may assume that $K=K_{\ab}$.

Then for all integers $n>0$ the group $G_n$ can be computed, as a subgroup of $(\Z/n\Z)^\times$. More precisely, given the conductor $f$ of $K$ and the group $G_f\leq (\Z/f\Z)^\times$, the group $G_n$ is the inverse image of the reduction of $G_f$ modulo $\gcd(n,f)$ by the map $(\Z/n\Z)^\times\to (\Z/\gcd(n,f)\Z)^\times$. Since $G_n$ is finite, there exists an algorithm which computes, for all integers $n>0$ and all prime powers $p^j\mid n$, the map $H^1(G_n,\mu_{p^j}) \to H^1(G_n,\mu_{p^{\ord_p(n)}})$ (\emph{i.e.} which computes the domain, the codomain, and the values of the map).

Then Lemma~\ref{lem:CB_algorithm} yields an algorithm to compute the Chevalley-Bass number of $K$: for each odd prime $p\mid \lambda$, and for $p=2$ if $\zeta_4\in K$, one runs (by increasing order) through integers $j$ from $\max\{\ord_p(\lambda),\ord_p(f)-\ord_p(\lambda)\}$ to $\ord_p(f)$. If $j=\ord_p(f)$ one has $\ord_p(\Lambda)=\ord_p(f)$ since this is the maximal value one can achieve. Otherwise, for each $t$ such that $j<t\leq \ord_p(f)$, check whether the map $H^1(G_n,\mu_{p^j}) \to H^1(G_n,\mu_{p^t})$ is surjective or not, where $n$ denotes the integer defined in Lemma~\ref{lem:CB_algorithm}. If one of these maps in not surjective, replace $j$ by $j+1$ and start again. If these maps are surjective for all relevant values of $t$, then the integer $j$ is the $p$-adic valuation of the Chevalley-Bass number.

When $\zeta_4\notin K$ and $\ord_2(f)\geq 3$, the same strategy works for $p=2$, but only the value $j=\ord_2(f)-1$ needs to be tested. Finally when $\ord_2(f)\leq 2$ we have that $\ord_2(\Lambda)=2$ by Theorem~\ref{thm:main}.

%%%%%%%%%%%%%%%%%%%%%%%%%%%%%%%%%%%%%%%%%%%%%

\subsection{Examples}
\label{sub:examples}

Throughout this section, we consider the following situation: $p$ is an odd prime number, and $m$ is a prime number such that $p^4$ divides $m-1$.
Consider a cyclic subgroup
$$
\Delta \leq \Omega_p^{2,4} \times (\Z/m\Z)^\times \leq (\Z/p^4m\Z)^\times
$$
generated by $(1+p^2,\gamma)$ where $\gamma\in(\Z/m\Z)^\times$ has order $p^2$, $p^3$ or $p^4$. Then we have
$$
\Delta \cap \left(\Omega_p^{2,4} \times \{1\}\right) = \{1\}.
$$
Under these conditions, the field
$$
K:=\Q(\zeta_{p^4m})^\Delta
$$
has conductor $p^4m$, and the number of roots of unity contained in $K$ is $2p^2$. According to Theorem~\ref{thm:main}, the Chevalley-Bass number $\Lambda$ of $K$ satisfies
$$
\Lambda \in \{4p^2, 4p^3, 4p^4\}.
$$
We shall show that the three cases occur, depending on the order of $\gamma$.

\subsubsection*{The case when $\Lambda$ is minimal}

Suppose that $\Delta$ is cyclic of order $p^2$, generated by $\tau=(1+p^2,\gamma)$ where $\gamma$ has order $p^2$. We claim that $\Lambda=4p^2$ in this case.

According to Lemma~\ref{lem:CB_algorithm}, the relevant values of $t$ are $3$ and $4$, and the corresponding values of $n$ are $p^3m$ and $p^4m$. So it suffices to check the surjectivity of the two maps
\begin{equation}
\label{eq:minimal_Lambda_maps}
H^1(G_{p^3m},\mu_{p^2}) \to H^1(G_{p^3m},\mu_{p^3}) \quad \text{and} \quad H^1(G_{p^4m},\mu_{p^2}) \to H^1(G_{p^4m},\mu_{p^4}).
\end{equation}

Firstly, we note that $G_{p^4m}=\Delta$, and this group is isomorphic (via the first projection) to $\Omega_p^{2,4}$, with its natural action on $\mu_{p^4}$. It follows from Lemma~\ref{lem:2b} that
$$
H^1(G_{p^4m},\mu_{p^4}) = H^1(\Omega_p^{2,4},\mu_{p^4}) = 0,
$$
hence the second map in \eqref{eq:minimal_Lambda_maps} is surjective for trivial reasons.

Take now $n=p^3m$, then the Galois group $G_{p^3m}$ is the image of $\Delta$ by the reduction map $(\Z/p^4m\Z)^\times\to (\Z/p^3m\Z)^\times$. Thus, $G_{p^3m}$ is the subgroup of $(\Z/p^3m\Z)^\times$ generated by the couple $(1+p^2,\gamma)$, where $1+p^2$ has now order $p$, and $\gamma$ is unchanged. Since $\gamma$ has order $p^2$, the norm $\norm$ of $(1+p^2,\gamma)$ is $p$ times the norm of $1+p^2$, hence is multiplication by $up^2$ (for some $u$ coprime to $p$), according to the calculation made in the proof of Lemma~\ref{lem:2b}. It follows that
$$
H^1(G_{p^3m},\mu_{p^3}) = \ker(\norm)/\img(p^2) = \frac{p\Z/p^3\Z}{p^2\Z/p^3\Z} \simeq \Z/p\Z.
$$
Since $G_{p^3m}$ acts trivially on $\mu_{p^2}$, we have $H^1(G_{p^3m},\mu_{p^2})=\Hom(G_{p^3m},\mu_{p^2})=\Z/p^2\Z$. 
To conclude, the short exact sequence $1\to \mu_{p^2}\to \mu_{p^3}\to \mu_p\to 1$ yields a long exact sequence of cohomology whose first terms are
$$
0\to \Z/p^2\Z=\Z/p^2\Z \to \Z/p\Z \to H^1(G_{p^3m},\mu_{p^2})=\Z/p^2\Z \to H^1(G_{p^3m},\mu_{p^3})=\Z/p\Z.
$$
(note that $G_{p^3m}$, like $\Delta$, fixes exactly $p^2$th roots of unity, which gives the first three $H^0$). The exactness of this sequence proves that the last map, which is none other than the first map in \eqref{eq:minimal_Lambda_maps}, is surjective, hence the result.

\subsubsection*{The case when $\Lambda$ is intermediate}

Suppose that $\Delta$ is cyclic of order $p^3$, generated by $\tau=(1+p^2,\gamma)$ where $\gamma$ has order $p^3$. We claim that $\Lambda=4p^3$ in this case.

Let us start with $j=2$, $t=4$ and $n=p^4m=f$, then $G_n=\Delta$. We claim that the map
$$
H^1(\Delta,\mu_{p^2}) \to H^1(\Delta,\mu_{p^4})
$$
is not surjective, which proves that $\ord_p(\Lambda)>2$. Indeed, the short exact sequence $1\to \mu_{p^2}\to \mu_{p^4}\to \mu_{p^2}\to 1$ yields a long exact sequence of cohomology whose first terms are
$$
0\to \Z/p^2\Z=\Z/p^2\Z \to \Z/p^2\Z \to H^1(\Delta,\mu_{p^2})=\Z/p^2\Z \to H^1(\Delta,\mu_{p^4}) = \Z/p\Z.
$$
By exactness, the last map has to be zero, hence is not surjective.

It remains to check that, for $j=3$, the relevant maps are surjective. In fact, $n=p^4m$ is the unique value of $n$ to consider. This case is easy since $G_n=\Delta$ again, and the short exact sequence $1\to \mu_{p^3}\to \mu_{p^4}\to \mu_p\to 1$ yields a long exact sequence of cohomology whose first terms are
$$
0\to \Z/p^2\Z=\Z/p^2\Z \to \Z/p\Z \to H^1(\Delta,\mu_{p^3})=\Z/p^2\Z \to H^1(\Delta,\mu_{p^4}) = \Z/p\Z.
$$
By exactness, the last map is surjective, hence the result.

\subsubsection*{The case when $\Lambda$ is maximal}

Suppose that $\Delta$ is cyclic of order $p^4$, generated by $\tau=(1+p^2,\gamma)$ where $\gamma$ has order $p^4$. We claim that $\Lambda=4p^4$ in this case. This can be checked in just one computation: the map
$$
H^1(\Delta,\mu_{p^3})=\Z/p^2\Z \to H^1(\Delta,\mu_{p^4}) = \Z/p^2\Z
$$
is not surjective. Details are left to the reader.

\subsection*{Acknowledgements}

The first author warmly thanks Yuri Bilu for introducing him to this problem, and for providing great feedback throughout the project. He also thanks Aurel Page for sharing insight on the algorithmic aspects of the problem.

After the manuscript was accepted for publication, the authors used Claude Opus 5.5 to proofread the text and identify minor errors. While doing so, Claude suggested an answer to Bilu's last question. Based on these suggestions, the authors stated and proved Corollary~\ref{cor:degree_bound}. They take full responsibility for the content of the publication.

%%%%%%%%%%%%%%%%%%%%%%%%%%%%%%%%%%%%%%%%%%%%%

%%%%%%%%%%%%%%%%%%%%%%%%%%%%%%%%%%%%%%%%%%%%%

\bibliographystyle{amsalpha}
\bibliography{biblioCB_corrected.bib}

%%%%%%%%%%%%%%%%%%%%%%%%%%%%%%%%%%%%%%%%%%%%%

%%%%%%%%%%%%%%%%%%%%%%%%%%%%%%%%%%%%%%%%%%%%%

\bigskip

\textsc{Jean Gillibert}, Universit\'e de Toulouse, Institut de Math{\'e}matiques de Toulouse, CNRS UMR 5219, 118 route de Narbonne, 31062 Toulouse Cedex 9, France.

\emph{E-mail address:} \texttt{jean.gillibert@math.univ-toulouse.fr}
\medskip

\textsc{Florence Gillibert}, \texttt{f.gillibert@yahoo.fr}
\medskip

\textsc{Gabriele Ranieri}, Istituto Superiore Sandro Pertini, viale Cavour, 267, 55100 Lucca, Italy.

\emph{E-mail address:} \texttt{g.ranieri@isipertinilucca.edu.it}

%%%%%%%%%%%%%%%%%%%%%%%%%%%%%%%%%%%%%%%%%%%%%

\end{document}